\documentclass[conference]{IEEEtran}
\IEEEoverridecommandlockouts

\usepackage{soul}
\usepackage{epsfig}
\usepackage{amsmath, amsfonts, environ, amsthm}
\usepackage[section]{placeins}
\usepackage{lipsum}
\usepackage{hyperref}
\usepackage{algorithm}
\usepackage{multicol}
\usepackage{cleveref}
\usepackage{bbold}
\usepackage{amsopn}
\usepackage{cite}
\usepackage{algorithmic}
\usepackage{graphicx}
\usepackage{textcomp}
\usepackage{xcolor}

\def\BibTeX{{\rm B\kern-.05em{\sc i\kern-.025em b}\kern-.08em
    T\kern-.1667em\lower.7ex\hbox{E}\kern-.125emX}}

\graphicspath{{Figures/}}

\newtheorem{definition}{Definition}[section]
\newtheorem{lemma}{Lemma}[section]
\newtheorem{theorem}{Theorem}[section]
\newtheorem{corollary}{Corollary}[section]

\begin{document}

\title{Block-missing data in linear systems: An unbiased stochastic gradient descent approach}

\author{
    \IEEEauthorblockN{Chelsea Huynh}
    \IEEEauthorblockA{Department of Mathematics\\ University of California, Irvine \\ Irvine, CA, USA}
    \and
    \IEEEauthorblockN{Anna Ma}
    \IEEEauthorblockA{Department of Mathematics\\ University of California, Irvine\\ Irvine, CA, USA}
    \and
    \IEEEauthorblockN{Michael Strand}
    \IEEEauthorblockA{Department of Statistics\\ University of California, Irvine\\ Irvine, CA, USA}
}

\maketitle

\begin{abstract}
Achieving accurate approximations to solutions of large linear systems is crucial, especially when those systems utilize real-world data. A consequence of using real-world data is that there will inevitably be missingness. Current approaches for dealing with missing data, such as deletion and imputation, can introduce bias. Recent studies proposed an adaptation of stochastic gradient descent (SGD) in specific missing-data models. In this work, we propose a new algorithm, $\ell$-tuple mSGD, for the setting in which data is missing in a block-wise, tuple pattern. We prove that our proposed method uses unbiased estimates of the gradient of the least squares objective in the presence of tuple missing data. We also draw connections between $\ell$-tuple mSGD and previously established SGD-type methods for missing data. Furthermore, we prove our algorithm converges when using updating step sizes and empirically demonstrate the convergence of $\ell$-tuple mSGD on synthetic data. Lastly, we evaluate $\ell$-tuple mSGD applied to real-world continuous glucose monitoring (CGM) device data.
\end{abstract}


\section{Introduction}

In today's data-driven world, analysts and scientists face one inevitable challenge when dealing with large-scale data -- real-world data goes missing. Current practices for handling missing data include deletion and imputation. However, these practices can reduce the statistical power of a study, produce bias, and lead to invalid conclusions when done naively. 

Extensive research has proposed sophisticated approaches for matrix completion and recovery that can address such issues around missing data trends and patterns. Some avenues can bypass the need to impute values altogether. For example, in a recent study from \cite{NMTMA-12-1}, Ma and Needell introduce an algorithm to circumvent the need to recover or impute the missing data when solving linear systems. Other works have extended their findings to a different and more robust version of what this algorithm accomplishes \cite{needell2014stochastic} in which, the authors propose an averaged stochastic gradient algorithm that adapts to data missing completely at random (i.e. each covariate has a different probability of being missing). The notion of structured, block-wise missingness in a features matrix is not a novel concept either. For example, Khayati, Cudré-Mauroux, and Böhlen \cite{Khayati2020} introduced an algorithm that considers variations in correlation when recovering missing blocks in time series data, while Xue and Qu \cite{blocks2} propose an imputation method for multi-source block-wise missing data. 

In real-world time series data, failures in sensors and communication can result in a loss of multiple consecutive data points, which creates structured ``blocks" or ``tuples" of missing data.  For example, suppose we have a data set where each row of the matrix corresponds to a sensor that sends information to a server to be processed into a predicted output value. Each column in the matrix corresponds to a record of information that the sensor is sending every 5 minutes (i.e. column 1 is the information sent at minute 5, column 2 is information sent at minute 10, etc.) If a sensor is broken for an extended period, multiple consecutive records will be missing in a block-like structure. Time series data can be used in the context of linear systems to predict certain metrics, but when data points are consecutively missing, one cannot expect to form accurate predictions.

In this paper, we provide an algorithm to estimate solutions to linear systems with block-structured missing data \textit{without} imputation. We propose an algorithm based on Stochastic Gradient Descent (SGD) with an update function that considers the system's structured block missingness and prove it's estimates are unbiased and that it will converge to a solution. Experiments are provided to demonstrate the performance of this algorithm in comparison to other methods of solving linear systems using synthetic data. Additionally, we utilize the proposed algorithm to solve a least squares problem using continuous glucose monitoring (CGM) data to demonstrate its applicability in a real-world setting.

\section{Background}

Before presenting our main contributions, we briefly overview SGD and the missing SGD (mSGD) algorithms. This is not an exhaustive literature review of both methods and is only meant to give the reader a sufficient background for the remainder of this paper. We refer the reader to~\cite{ruder2016overview, bottou2010large, needell2014stochastic} for more details on these methods.

\subsection{Stochastic Gradient Descent}

Consider a linear system $Ax_\star = y$ where $A \in \mathbb{R}^{m \times n}$ and $m \gg n$. The goal is to, given $A$ and $y$, find $x_\star$. SGD can be used to accomplish this task by minimizing the least squares objective function 
\begin{align*}
    F(x) := \frac{1}{2m}||Ax - y||^2 = \frac{1}{m}\sum_{i=1}^m f_i(x),
\end{align*}
through iterative updates $x_{k+1} = x_k -\alpha_k \nabla f_{i} (x_k)$. Here, $\nabla f_{i} (x_k)$ denotes the evaluation at $x_k$ of the gradient of $f_{i}$ with respect to $x$, $\alpha_k$ is the step size at the $k^{th}$ iteration, and $i$ is randomly selected according to some distribution (e.g., uniformly over all $i = 1, 2,..., m$).  The algorithm minimizes $F(x)$ using unbiased estimates of the gradient, i.e., $f_i(x)$ such that $\mathbb{E}\left[\nabla f_i(x)\right] = \nabla F(x)$ at every iteration. 

\subsection{SGD for Missing Data}
\label{ssec:msgd}

Conventional optimization methods may not produce accurate solutions if data is absent from the matrix $A$. In particular, standard applications of SGD to least squares problems assume that $A$ is fully available, i.e., no data is missing. 

Instead of naively imputing, estimating, or deleting values, Ma and Needell \cite{NMTMA-12-1} propose the mSGD method to solve linear systems when data is assumed to be missing completely at random. In particular, their approach assumes the availability of the entries of $A$ can be modeled as i.i.d Bernoulli random variables with probability of being present $p$. Under this assumption, an adapted SGD update function can be furnished with a correction term to ensure iterates move in the direction of the gradient of the least squares objective in expectation. The algorithm proposed in \cite{NMTMA-12-1} is $x_{k+1} = x_k - \alpha_k h(x_k)$, with 
\begin{align}
\label{eq:msgd} 
    h(x) := \frac{1}{p^2}(\tilde{A}_i^T (\tilde{A}_i x - p y_i)) - \frac{1-p}{p^2} \text{diag}(\tilde{A}_i^T \tilde{A}_i) x,
\end{align}
where $\tilde{A}$ denotes a matrix containing only accessible data (data that is not missing from $A$). Note that the latter part of \Cref{eq:msgd} is the ``correction" term which, informally, takes into account the fact that there are missing entries, and when $p=1$ we recover the original SGD iterate.

\section{Main method}
\label{sec:main}

Let $A^{(\ell)}_{i,j}$ for $i = 1,2,..., m$ and $j = 1,2,...,n\slash \ell$ denote the $j^{th}$ tuple of size $\ell$ in row $i$ where the tuple contains elements of columns $(j-1)\ell + 1$ to $j\ell$. With this notation, we now establish a rigorous definition for our structured missing data model.

\begin{definition}[$\ell$-tuple missingness]
    \label{def:tuple}
     Let $A \in \mathbb{R}^{m\times n}$ be a fixed matrix and $\ell$ be a positive, non-zero integer that divides $n$. Consider the $i^{th}$ row of ${A}$ partitioned into $n/\ell$ sections, or `tuples', of length $\ell$. The incomplete realization of $A$, denoted $\tilde{A}$, has ``$\ell$-tuple missing data" or ``tuple-missing data" if for $i = 1,2,..., m$ and $j = 1,2,...,n\slash\ell$, $\tilde{A}^{(\ell)}_{i, j}$ are i.i.d. Bernoulli random variables with 
    \begin{equation}
        \tilde{A}_{i,j}^{(\ell)} = \begin{cases}
        A_{i,j}^{(\ell)} &\text{with probability } p\\
        0^{(\ell)} & \text{with probability }1-p,
        \end{cases}
    \label{eq:binary_mask}
    \end{equation}
    where $A^{(\ell)}_{i,j}$ is the $j^{th}$ tuple of $A_i$ and $0^{(\ell)}$ is a tuple of zeroes.
\end{definition}

In the missing $\ell$-tuples model, missingness occurs across rows, independently at random in $\ell$ consecutive entries simultaneously. This definition is amendable to matrices $A$ with infinite rows, as would be the case in a streaming data setting. However, for the remainder of this paper, we assume a finite-row setting for the sake of simplicity and assume $A\in\mathbb{R}^{m\times n}$ for $n,m< \infty$. It is also straightforward to extend the missing tuples model and our algorithm to allow for different tuple lengths and probabilities within a single row but, again for simplicity, we will only consider the case where $\ell$ and $p$ are fixed over $A$.
\Cref{fig:tuple} is a visual example of the tuple-missing model. 

\begin{figure}[!h]
    \centering
    \includegraphics[scale=0.6]{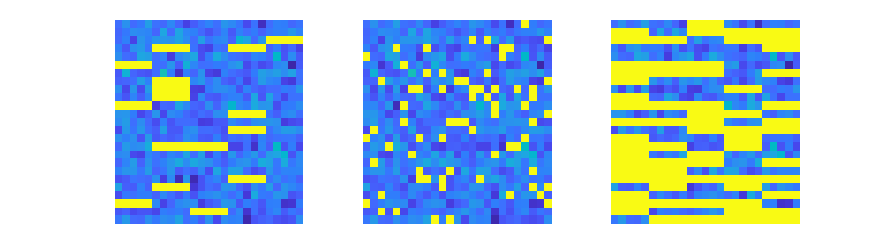}
    \caption{A heat map of the same matrix with different types of tuple-missing data. Here, yellow spaces indicate missing information. Every natural number $\ell \mid n$ is a valid length under the missing-tuples model and we make the distinction that $\ell = 1$ is synonymous with i.i.d Bernoulli missing individual entries, as is the case in the mSGD model. \textbf{Left:} $\ell = 5, p = 0.9$ \textbf{Center:} $\ell = 1, p = 0.9$, \textbf{Right:} $\ell = 5, p = 0.5$.}
    \label{fig:tuple}
\end{figure}

We model the matrix structure of tuple-missingness through a \textit{binary mask} $D\in\{0,1\}^{m \times n}$, which we use to specify the tuple-missingness of a fixed matrix $\tilde{A}$ through $\tilde{A} = D \,\odot A$ where $\odot$ is the Hadamard, or element-wise product. In other words, $D$ indicates the positions of missing data in $A$, with zeroes, and observed data, with ones. 

\noindent\subsection{Main Result}

We now turn towards the construction of a stochastic iterative algorithm to approximate solutions to tuple-missing systems. Suppose that we want to find the least squares solution to the linear system $A x_\star =y$ but only have access to an incomplete realization of $A$, $\tilde{A}$, which has tuple-missing data as reflected in \Cref{def:tuple}. In such a case, we consider the objective function:
\begin{align*}
    \tilde{F}(x) := \frac{1}{2m}||\tilde{A}x-y||^2 = \frac{1}{m}\sum_{i=1}^m\tilde{f}_i(x),
\end{align*}
and
\begin{align*}
    \tilde{f}_i(x) = \frac{1}{2}\left(\tilde{A}_ix-y_i\right)^2,
\end{align*}
for which $y_i,\tilde{A}_i$ is the $i^{th}$ element of $y$ and the $i^{th}$ row of $\tilde{A}$, respectively. Naively applying SGD to the above objective is unlikely to result in the least squares solution when using the completed matrix $A$. One will find taking the expectation of $\nabla\tilde{f}_i$ with respect to the binary mask $D$ and the uniformly chosen row $i$ does not yield iterates that move in the direction steepest descent when the matrix contains tuple-missing data. That is to say, if we denote the expectation with respect to all possible binary masks $D$ as $\mathbb{E}_\delta$ and for the random row choice $\mathbb{E}_i$, then, assuming these two sources of randomness are independent,
\begin{align*}
   \mathbb{E}\left[\nabla\tilde{f}_i\right] =  \mathbb{E}_{i,\delta}\left[\nabla\tilde{f}_i\right] = \mathbb{E}_i\mathbb{E}_\delta\left[\nabla\tilde{f}_i\right] \neq \nabla F(x).
\end{align*}
However, our proposed algorithm $\ell$-tuple mSGD adds a correction term to the SGD iterate which accounts for the tuple missing data structure. In our proposed algorithm we take the step update $x_{k+1} = x_k - \alpha_k h_\ell(x_k)$, for some learning rate $\alpha_k$ and
\begin{align}
    h_\ell(x) := \frac{1}{p^2}\left(\tilde{A}_i^{T}\left(\tilde{A}_i x - p y_i\right)\right) -\frac{1-p}{p^2} \,\, L \odot \tilde{A}_i^{T}\tilde{A}_i x.\label{eq:hx}
\end{align}
In \Cref{eq:hx}, $L$ is a fixed matrix, dependant only on the tuple length $\ell$, which captures the tuple structure of the matrix outer product $\tilde{A}_i^T\tilde{A}_i$. We define $L$ as the $n \times n$ matrix of all zeroes except for the $\ell\times\ell$ block-diagonal of ones. The matrix $L$ works within $h_\ell(x)$ to adjust the expectation of the iterates back in the direction of the objection function specifically according to the influence of the tuples model.

\begin{theorem}\label{thm:main}
    Let $Ax_\star = y$ be a consistent overdetermined linear system. Assume that $y \in \mathbb{R}^m$ is given but $A \in \mathbb{R}^{m \times n}$ is only partially known. In particular, suppose that one only has access to $\tilde{A}$ as defined in \Cref{def:tuple}. Denoting the least squares objective $F(x) = \frac{1}{2m}\| Ax - y\|^2$ we have that:
    \begin{displaymath}
        \mathbb{E}[h_\ell(x)]= \nabla F(x),
    \end{displaymath}
    where $h_\ell(x)$ is as defined in \Cref{eq:hx} and the expectation is taken with respect to data missingness and random row selection. 
\end{theorem}

\begin{proof}
    We start with the expectation of \Cref{eq:hx}:
    \begin{align}
        &\mathbb{E}\left[h_\ell(x)\right]\nonumber\\
        &= \mathbb{E}_i\mathbb{E}_\delta\left[\frac{1}{p^2}\left(\tilde{A}_i^{T}\left(\tilde{A}_i x - p y_i\right)\right) -\frac{1-p}{p^2} \,\, L \odot \tilde{A}_i^{T}\tilde{A}_i x\right]\nonumber\\
        &= \frac{1}{p^2}\mathbb{E}_i\mathbb{E}_\delta\left[\tilde{A}_i^{T}\tilde{A}_i x\right] - \frac{1}{p}\mathbb{E}_i\mathbb{E}_\delta\left[\tilde{A}^T_iy_i\right]\nonumber\\ 
        &- \frac{1-p}{p^2}\mathbb{E}_i\mathbb{E}_\delta\left[L \odot \tilde{A}_i^{T}\tilde{A}_i x\right] .\label{eq:full_exp}
    \end{align}
    Now, we evaluate the three expectations comprising \Cref{eq:full_exp}. 
    
    Firstly, we view the outer product of $\tilde{A}_i$ with itself as a matrix of outer products between constituent tuples, $\left(\tilde{A}_{i,j}^{(\ell)}\right)^T\tilde{A}_{i,k}^{(\ell)}\in\mathbb{R}^{\ell\times \ell}$ for $j,k = 1,2,...,n/ \ell$. The resulting expectation of this outer product with respect to the missingness has off-block diagonals which are scaled with a factor of $p^2$ and the block diagonal matrices are scaled by a factor of $p$. In particular, the support of the matrix $L$ corresponds to the elements in which $\mathbb{E}_\delta\left[\tilde{A}_i^T\tilde{A}_i\right]$ is scaled by a factor of $p$. One can also obtain the support of the elements that are scaled by a factor of $p^2$ by taking $\mathbb{1} - L$ where $\mathbb{1}$ denotes a $n \times n$ ones matrix. 
    \begin{align}
        &\frac{1}{p^2}\mathbb{E}_i\mathbb{E}_\delta\left[\tilde{A}_i^{T}\tilde{A}_i x\right]\nonumber\\
        &= \frac{1}{p^2}\mathbb{E}_i\left[p^2(\mathbb{1} - L)\odot A^T_iA_ix + pL\odot A^T_iA_ix\right]\nonumber\\
        &= \left(\frac{1}{p^2m}\right)\sum_{i=1}^mp^2A^T_iA_i x - p^2 L\odot A^T_iA_ix + pL\odot A^T_iA_ix\nonumber\\
        &= \left(\frac{1}{p^2m}\right)\left(p^2A^TAx + (p-p^2)L\odot A^TAx\right). \label{eq:outer_exp_taken}
    \end{align}
    For the second term of the expectation in \Cref{eq:full_exp}, it is a simple calculation to see that $\mathbb{E}_\delta\left[\tilde{A}^T_i\right] = pA^T_i$ which implies
    \begin{align}
         \frac{1}{p}\mathbb{E}_i\mathbb{E}_\delta\left[\tilde{A}^T_iy_i\right] &= \frac{1}{p}\mathbb{E}_i\left[pA^T_iy_i\right] = \frac{1}{m}\sum_{i=1}^mA^T_iy_i\nonumber\\  
         &= \frac{1}{m}A^Ty.\label{eq:trans_exp_taken}
     \end{align}
    Using \Cref{eq:outer_exp_taken} and \Cref{eq:trans_exp_taken} we complete \Cref{eq:full_exp} as follows:
    \begin{align}
        &\frac{1}{p^2}\mathbb{E}_i\mathbb{E}_\delta\left[\tilde{A}_i^{T}\tilde{A}_i x\right] - \frac{1}{p}\mathbb{E}_i\mathbb{E}_\delta\left[\tilde{A}^T_iy_i\right] - \frac{1-p}{p^2}\mathbb{E}_i\mathbb{E}_\delta\left[L \odot \tilde{A}_i^{T}\tilde{A}_i x\right] \nonumber\\
        &= \left(\frac{1}{p^2m}\right)\left(p^2A^TAx-p^2L\odot A^TAx + pL\odot A^TAx\right)- \frac{1}{m}A^Ty\nonumber\\
        &- \left(\frac{1-p}{p^2m}\right) L\odot \left(p^2A^TAx-p^2L\odot A^TAx + pL\odot A^TAx\right)\nonumber\nonumber\\
        &= \left(\frac{1}{m}\right)\left(A^TAx - A^Ty\right)\nonumber\\
        &= \nabla F(x).\nonumber
    \end{align}
    Thus, we have that $ \mathbb{E}\left[h_\ell(x)\right] = \nabla F(x)$.
\end{proof}

Convergence of our algorithm for an updating step size can be shown by adopting a similar argument as was used in \cite{NMTMA-12-1} and the additional assumption that iterates are projected onto some convex domain $\mathcal{W}$ containing the solution $x_\star$ each step. 

We begin by citing a result from \cite{lem1} which guarantees convergence of SGD when using unbiased estimates of the gradient under some assumptions on the objective. 
\begin{theorem}[\hspace{-0.03ex}\cite{lem1}, Theorem 1] 
    \label{thm:bound}
    Let $F(x)$ be $\mu$-strongly convex and $g(x)$ a function having the properties $\mathbb{E}[g(x)] = \nabla F(x)$ and $\mathbb{E}\left[\left\| g(x) \right\|^2\right] \leq G$ for all $x \in \mathcal{W}$. If step size $\alpha_k = \frac{1}{\mu k }$ is chosen to update $x_{k+1} = \mathcal{P}_\mathcal{W}\left(x_k -\alpha_k g(x_k)\right)$, it follows that 
    \begin{align*}
        \mathbb{E}\left[F(x_{k+1}) - F(x_\star)\right] \leq \frac{17G(1+\log{(k))}}{\mu k},
    \end{align*}
    where $\mathcal{P}_\mathcal{W}$ is a projection function onto $\mathcal{W}$.
\end{theorem}  

\begin{corollary}
    \label{thm:convergence}
    Let $\tilde{A}$ be the tuple-missing analog of $A$ according to \Cref{def:tuple}, $\mu$ be the strong convexity constant of the objective function $F(x)$, and $G$ be the uniform upper bound of $\mathbb{E}[\| h_\ell(x)\|^2]$. Then choosing step size $\alpha_k = \frac{1}{\mu k}$, the $\ell$-tuple mSGD algorithm converges with expected error
    \begin{align*}
        \mathbb{E}\left[\left\| x_{k+1} - x_\star\right\|^2\right] \leq \frac{17G(1 + \log{(k)})}{\mu^2 k}.
    \end{align*}
\end{corollary}

The convergence of our algorithm is guaranteed by Theorem~\ref{thm:convergence} once we show that the assumptions are satisfied. Since Theorem~\ref{thm:main} guarantees unbiased estimates of $\nabla F(x)$ and $F(x)$ is $\mu$-strongly convex where $\mu$ is the square of the smallest singular value of $A$ divided by $m$~ (as shown, for e.g., in \cite{NMTMA-12-1}), it remains to prove that the expected squared norm of the $h_\ell(x)$ is uniformly bounded.
 
\begin{lemma}
    \label{lem:G}
    $\mathbb{E}\left\| h_\ell(x) \right\|^2 \leq G$, where
    \begin{align*}
        G =\frac{2B}{mp^3}\sum_i\left\| A_{i}\right\|^4 + \frac{2}{mp}\sum_iy_i^2\left\| A_i\right\|^2
    \end{align*}
    and $B = \underset{x \in \mathcal{W}}{\max}\left\|x\right\|^2$.
\end{lemma}
\begin{proof}
    \small
    \begin{align}
        &\mathbb{E}\left[\left\| h_\ell(x)\right\|^2\right]\nonumber\\ &= \mathbb{E}_{i,\delta}\left[\left\| \frac{1}{p^2}\left(\tilde{A}_i^T(\tilde{A}_ix - py_i) - (1-p)L\odot (\tilde{A}_i^T\tilde{A}_i)x\right)\right\|^2\right]\nonumber\\
        &\leq \frac{2}{p^4}\mathbb{E}_{i,\delta}\left\| \tilde{A}_i^T\tilde{A}_ix - (1-p)L\odot (\tilde{A}_i^T\tilde{A}_i)x\right\|^2 + \frac{2}{p^4}\mathbb{E}_{i,\delta}\left\| py_i\tilde{A}^T_i\right\|^2\nonumber\\
        &\leq \frac{2}{p^4}\mathbb{E}_{i,\delta}\left\| \tilde{A}_i\right\|^4\left\| x \right\|^2 + \frac{2}{p^2}\mathbb{E}_{i,\delta}\left\| y_i\tilde{A}_i\right\|^2.\label{eq:G_tot}
    \end{align}
    \normalsize
    Above we used the fact that
    \begin{align*}
        &\left\|  \tilde{A}_i^T\tilde{A}_i - (1-p)L\odot\tilde{A}_i^T\tilde{A}_i\right\|\\ &\leq      \left\|  \tilde{A}_i^T\tilde{A}_i - (1-p)L\odot\tilde{A}_i^T\tilde{A}_i\right\|_F\\ 
        &\leq \left\| \tilde{A}_i^T\tilde{A}_i\right\|_F = \left\| \tilde{A}_i\right\|^2
    \end{align*}
    where $\left\| \cdot \right\|_F$ denotes a Frobenius norm operation. We complete the bound by taking the expectations of these norms. To begin, we first observe that the square norm of a row $A_i$ can be written as the sum of square tuple-norms:
    \begin{align}
        \frac{2}{p^2}\mathbb{E}_{i,\delta}\left\| y_i\tilde{A}_i\right\|^2  &=  \frac{2}{p^2}\mathbb{E}_{i,\delta}\left[y_i^2\sum_{j=1}^{n\slash\ell}\left\| \tilde{A}_{i,j}^{(\ell)}\right\|^2\right]\nonumber\\
        &=  \frac{2}{p^2}\mathbb{E}_i\left[y_i^2\sum_{j=1}^{n\slash\ell}\mathbb{E}_\delta\left\| \tilde{A}_{i,j}^{(\ell)}\right\|^2\right]\nonumber\\
        &= \frac{2}{p}\mathbb{E}_i\left[y_i^2\sum_{j=1}^{n\slash\ell}\left\| A_{i,j}^{(\ell)}\right\|^2\right]\nonumber\\
        &= \frac{2}{p}\mathbb{E}_i\left[y_i^2\left\| A_i\right\|^2\right]\nonumber\\
        &=\frac{2}{mp}\sum_iy_i^2\left\| A_i\right\|^2,\label{eq:part1}
    \end{align}
    since $\left\| \tilde{A}_{i,j}^{(\ell)}\right\|^2$ is the norm of the $j^{th}$ tuple $A_{i,j}^{(\ell)}$ of the original matrix $A$ with probability $p$ or $0$ otherwise. The other term follows similarly by expanding the square and applying linearity of expectation:
    \begin{align*} 
        &\mathbb{E}_{i,\delta}\left\| \tilde{A}_i\right\|^4 = \mathbb{E}_{i,\delta}\left[\sum_{j=1}^{n\slash\ell}\left\| \tilde{A}_{i,j}^{(\ell)}\right\|^2\right]^2\\
        &= \mathbb{E}_i\left[\sum_{j=1}^{n\slash\ell}\mathbb{E}_\delta\left\| \tilde{A}_{i,j}^{(\ell)}\right\|^4 + \sum_{j\neq k}^{n\slash\ell}\mathbb{E}_\delta\left(\left\| \tilde{A}_{i,j}^{(\ell)}\right\|^2\left\| \tilde{A}_{i,k}^{(\ell)}\right\|^2\right)\right]\\
        &=\mathbb{E}_i\left[\sum_{j=1}^{n\slash\ell}\mathbb{E}_\delta\left\| \tilde{A}_{i,j}^{(\ell)}\right\|^4 + \sum_{j\neq k}^{n\slash\ell}\mathbb{E}_\delta\left\| \tilde{A}_{i,j}^{(\ell)}\right\|^2\mathbb{E}_\delta\left\| \tilde{A}_{i,k}^{(\ell)}\right\|^2\right].
    \end{align*}
    The last step follows from the independence assumed between different tuples. Then, by taking the expectation with respect to the binary masks, we obtain
    \begin{align}
        &\mathbb{E}_i\left[p\sum_{j=1}^{n\slash\ell}\left\| A_{i,j}^{(\ell)}\right\|^4 + p^2\sum_{j\neq k}^{n\slash\ell}\left\| A_{i,j}^{(\ell)}\right\|^2\left\| A_{i,k}^{(\ell)}\right\|^2\right]\nonumber\\
        &\leq p\mathbb{E}_i\left[\sum_{j=1}^{n\slash\ell}\left\| A_{i,j}^{(\ell)}\right\|^4 + \sum_{j\neq k}^{n\slash\ell}\left\| A_{i,j}^{(\ell)}\right\|^2\left\| A_{i,k}^{(\ell)}\right\|^2\right]\nonumber\\
        &= p\mathbb{E}_i\left\| A_{i}\right\|^4\nonumber\\
        &= \frac{p}{m}\sum_i\left\| A_{i}\right\|^4.\label{eq:part2}
    \end{align}
    The final result is obtained from substituting \Cref{eq:part1} and \Cref{eq:part2} into \Cref{eq:G_tot} and bounding $x\in\mathcal{W}$.
\end{proof}

\subsection{Relation to mSGD}
\label{ssec:relation_msgd}

Taking the tuple size $\ell = 1$, we exactly recover the mSGD update function $h(x)$ from \Cref{eq:msgd}. i.e., $h_1(x) = h(x)$. Furthermore, when $p$ is fixed, we expect the same amount of missingness (total entries missing) in the i.i.d tuple-missing model for any $\ell$ as we do in the mSGD model with i.i.d Bernoulli missing individual entries. Despite this, knowing the structure of the missingness \textit{does} give an advantage in solving systems. To see this, consider the difference in the expectation of the $\ell$-tuple mSGD and mSGD update, \Cref{eq:hx} and \Cref{eq:msgd} respectively, under the tuple-missing model:
\begin{align*}
    &\mathbb{E}_{i,\delta}\left[h_\ell(x) - h(x) \right]\\ &= \mathbb{E}_{i,\delta}\left[-\frac{1-p}{p^2} \,\, L \odot \tilde{A}_i^{T}\tilde{A}_i x+\frac{1-p}{p^2} \,\,\text{diag}\left( \tilde{A}_i^{T}\tilde{A}_i \right)x \right]\nonumber\\
    &= (1-p)A^TA\odot(I - L)x,\label{eq:bias}
\end{align*}
where $I$ is the $n\times n$ identity matrix. If $\ell = 1$ then $L = I$ and the above is $0$. We consider $ (1-p)A^TA\odot(I - L)$ to be a bias term introduced by the tuple structure which is not removed by mSGD.

It was shown in \cite{NMTMA-12-1} that mSGD recovers SGD when $p = 1$. The same is true for $\ell$-tuple mSGD. When $p = 1$, $\tilde{A}_i = A_i$ and $\dfrac{1-p}{p^2} \,\, L \odot \tilde{A}_i^{T}\tilde{A}_i x = 0$, such that $h_\ell(x)$ becomes $A_i^T\left(A_ix - y_i\right)$; the familiar SGD update. 

\section{Experimental results}
\label{sec:experiments}

The following experiments demonstrate the behavior of $\ell$-tuple mSGD on real and synthetic data. In order to simulate missingness according to Definition~\ref{def:tuple}, the complete data matrix $A$ is available in each of our experiments and has standard Gaussian distributed entries. Tuple-missing data is simulated by generating a new binary mask row each iteration to generate $\tilde{A}_i = D_i\odot A_i$. We report $\left\| x_{k} - x_\star \right\|^2$ as the error. 

\begin{figure}[!h]
    \includegraphics[scale = 0.6]{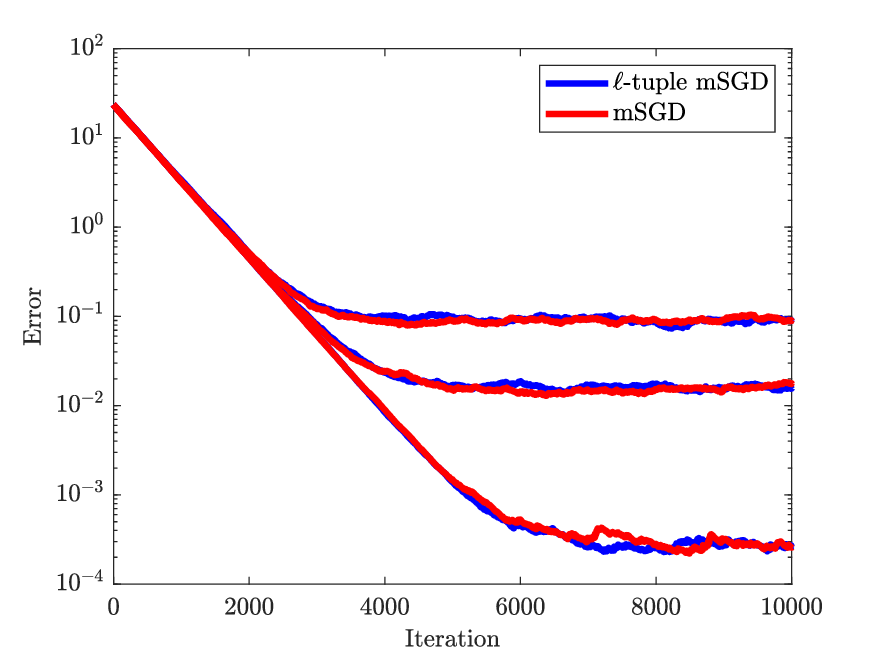}
    \caption{A comparison of mSGD and $\ell$-tuple mSGD for when $\ell = 1$ and for varying $p = 0.8, 0.95$ and $0.999$. Predictably, we see lower errors for higher values of $p$ in each algorithm.  Under these conditions, $p$ has the same meaning in the context of each separate algorithm.}
    \label{fig:ell_equal_1}
\end{figure}

In the first experiment, we compare mSGD to $\ell$-tuple mSGD in a sanity check where we speculate that both algorithms will produce similar errors given the correct chosen tuple length. In order to match the $\ell$-tuple mSGD model to mSGD's model, we must set $\ell=1$. In this case, the $L$ matrix is the identity matrix, implying $L\odot \tilde{A}_i^T\tilde{A}_i = \text{diag}\left(\tilde{A}_i^T\tilde{A}_i\right)$. For this experiment, $A \in \mathbb{R}^{m \times n}$ where $m = 10,000$ and $n = 25$. \Cref{fig:ell_equal_1} displays 20-experiment averaged results for increasing $p$ with fixed step size $\alpha_k = 10^{-3}$ and, as predicted, the results and behaviors of both algorithms are comparable.

\begin{figure}[!h]
    \includegraphics[scale = 0.6]{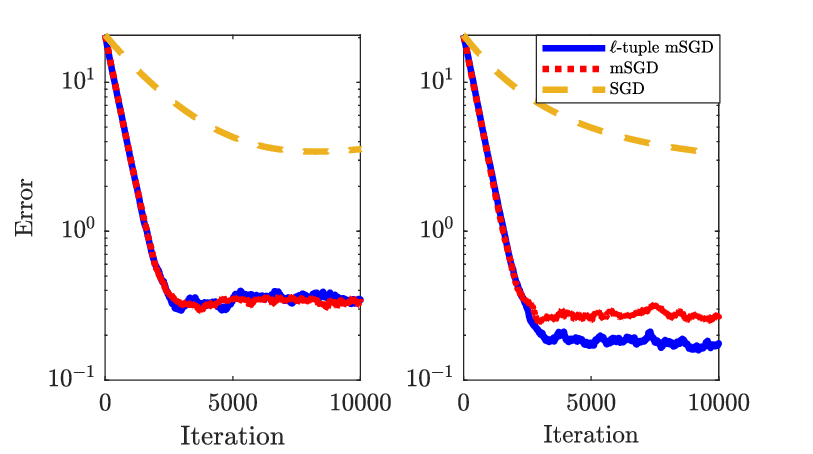}
    \caption{A comparison of $\ell$-tuple mSGD, mSGD, and SGD on different length tuple missing systems. \textbf{Left: } $\ell = 2$. \textbf{Right: } $\ell = 15$.}
    \label{fig:3_alg_comp}
\end{figure}

\Cref{fig:3_alg_comp} compares the average of 20 experiments of $\ell$-tuple mSGD against mSGD and SGD on $\tilde{A}$ having $m=8,000$ rows and $n=30$ columns for fixed $p=0.6$ and with $\ell$-tuple missing data for tuple length parameter $\ell\in \{2, 15\}$. With the binary mask drawn to fit the $\ell$-tuple-missingness model described in \Cref{sec:main}, we can see that mSGD and $\ell$-tuple mSGD achieves a smaller convergence horizon than SGD. Thus, mSGD and $\ell$-tuple mSGD are more desirable when finding approximate solutions to random missing data and structured random missing data, respectively. Furthermore, we observe the resulting difference in error convergence horizons due to the bias term mentioned in \Cref{ssec:relation_msgd} which is greater for larger values of $\ell$.

\begin{figure}[!h]
    \includegraphics[scale=0.6]{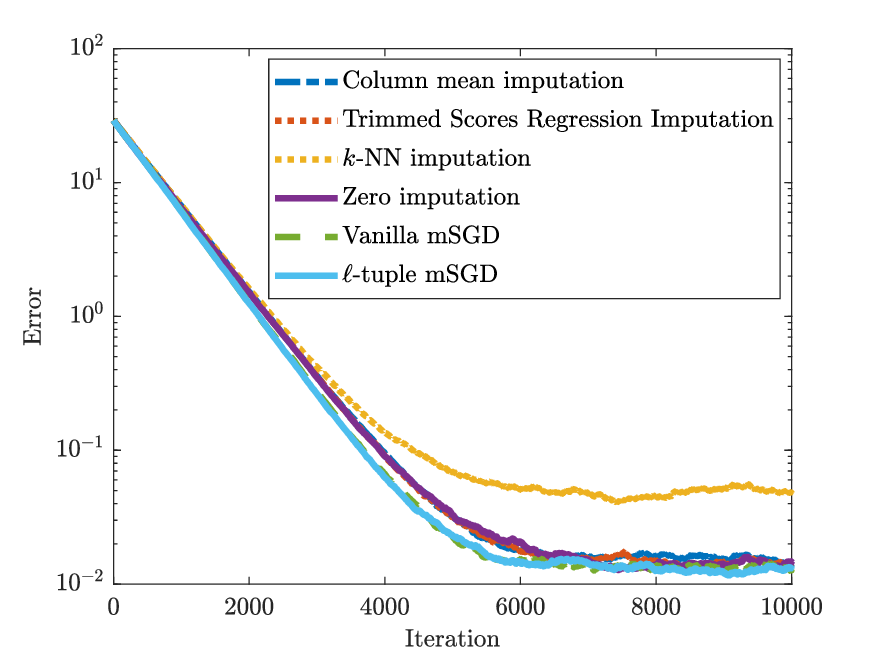}
    \caption{Comparison of convergence of other missing data methods and $\ell$-tuple mSGD in the presence of tuple missing data.}
    \label{fig:missing_data_comps}
\end{figure}

\Cref{fig:missing_data_comps} compares $\ell$-tuple mSGD with other popular imputation and missing data methods on tuple-missing data averaged over 20 experiments. Each of the algorithms performs SGD after applying its missing data correction. For example, column means imputes the column averages from non-missing entries in the same column. $k$-NN imputes the average of the 5 nearest non-missing neighbors to missing entries (under the assumption that columns of $A$ follow the same distribution). Trimmed scores regression\footnote{Using the Matlab package Missing Data Imputation Toolbox \cite{tsr}.} calculates imputations through the regression equation and a trimmed-scores matrix for the non-missing entries. Here, we see that $\ell$-tuple mSGD is able to obtain a convergence horizon comparable to other methods. While this experiment was only performed on a synthetic linear system, this result shows that, under the right conditions, $\ell$-tuple mSGD fits in with other contemporary missing data methods in reducing error. To produce this figure, we generated $A \in \mathbb{R}^{m \times n}$ such that $m = 10,000, n = 100, \ell = 50$, $p = 0.95$ and $\alpha_k = 8\cdot 10^{-4}$.

\section{Application - \textit{Continuous Glucose Monitoring}}
\label{sec:cgm}

We now present an application of $\ell$-tuple mSGD in predicting patient blood glucose levels using continuous glucose monitoring (CGM) device data. CGM is used in medical devices commonly worn by diabetic patients to monitor their blood glucose in real-time. It has been claimed to be beneficial in the early detection and prevention of hypoglycemic events \cite{cgmclaim}. While biometric data is continuously collected, the invasive nature of some wearable CGM devices creates the potential for noise \cite{cgmdatapaper} in the sensor readings, which lessens the accuracy in estimating the patient's glucose at these points in time. If CGM data points are considered ``missing" according to a specified acceptable noise threshold, then we hypothesize that $\ell$-tuple mSGD would make more accurate glucose level predictions than SGD or mSGD in this context.

For the following application, we used biometric data \cite{cgmdata} collected over three consecutive days from a single patient to construct a system of equations $Cx = g$ in which each of the resulting $365$ rows of $C$ contains 5 minutes of sequential groups of feature readings, taken each second, that correspond to one glucose level reading in $g$. In this setting, $C$ has a tuple structure where $\ell$, the tuple length, is the fixed number of biometric features per reading. If we use a noise-level-based missingness threshold, then this missingness takes on the same tuple structure, and we can use $\ell$-tuple mSGD to find the best value for $x$. The procedure and results of constructing $C$ in this way are available in our \href{https://github.com/mstrand1/l-tuple-mSGD}{GitHub repository}\footnote{https://github.com/mstrand1/l-tuple-mSGD}.

\begin{figure}[!h]
    \includegraphics[scale=0.6]{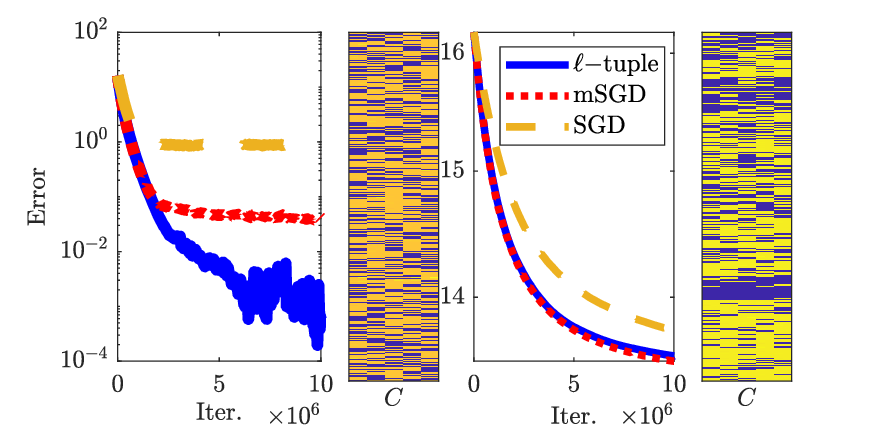}
    \caption{Comparison of $\ell$-tuple mSGD, mSGD, and SGD on the CGM data system $Cx_\star = \hat{g}$ for $C\in\mathbb{R}^{365\times 10}$ with $\ell = 2$ and a noise-threshold that rejects about 40\% of data as ``too-noisy". Here, $C$ has $10$ columns since we choose 5 readings of 2 features every 5 minutes. Two ways of simulating missingness are compared and visualized to the right of their respective error plots: \textbf{Left:} Synthetic i.i.d Bernoulli tuple missing \textbf{Right:} Noise-threshold missing.}
    \label{fig:cgm}
\end{figure}

In \Cref{fig:cgm}, we see the results of using two biometric features, ``ECGWaveform" and ``ECGAmplitude"\footnote{The choice of these electrocardiogram (ECG) variables is motivated by previous works in predicting hypoglycemia using biosensor data collected by wearable devices \cite{cgmdl, cgmecg}.}, with 5 readings each per every 5 minutes to populate $C$ and $\ell$-tuple mSGD, mSGD, and SGD to approximate a vector $x$ which minimizes glucose level prediction error. Noise thresholds are determined by each reading's corresponding ``ECGnoise" variable. To satisfy the assumption that the given linear system is consistent, we use the data-dependent right side vector $\hat{g} = CC^\dag g$. Using $\hat{g}$ instead of $g$ ensures that the system $Cx = \hat{g}$ has a solution. When applying these methods to the linear system $Cx = \hat{g}$ with noise-threshold missing data, we found that all three methods behave similarly and attain the same convergence horizon \Cref{fig:cgm} (right). We conjecture that the noise from the inconsistent system may be greater than the noise due to missing data and is thus the dominating factor in the empirical behavior of these iterative methods. We leave exploring the trade-off and connection between noise and missing data for future work. 

\section{Conclusions}
\label{sec:conc}

In this work, we proposed an SGD-type algorithm that utilized unbiased estimates of the gradient of the least squares objective in the presence of tuple missing data. We proved that, given an updating step size, the algorithm will converge to an error horizon bound by some constant.

The experiments on real and synthetic data support our claims and show empirically that $\ell$-tuple mSGD converges to error horizons comparable to naively-applied contemporary missing data methods such as imputation and mSGD in cases where $\ell > 1$. Therefore, $\ell$-tuple mSGD proves to be an effective choice for solving linear systems with missing data without introducing bias or expensive imputation procedures to minimize approximation errors.

\section*{Acknowledgment} We would like to thank Meha Patel, Sam Rath, and Chupeng Zheng for their involvement and collaboration in facilitating the ideas and conversations around this project. 

\bibliographystyle{ieeetr}
\bibliography{ref}

\end{document}